\documentclass[reqno,12pt]{amsart}

\usepackage{amsmath,amssymb,mathtools,verbatim,slashed,upgreek,bbm,appendix}
\usepackage[protrusion=true,babel=true]{microtype}
\usepackage[english]{babel}
\usepackage[widespace]{fourier}
\usepackage[backrefs]{amsrefs}
\usepackage[margin=1in]{geometry}
\usepackage[onehalfspacing]{setspace}
\usepackage[pdfusetitle,pagebackref]{hyperref}
\numberwithin{equation}{section}                
\usepackage[nameinlink,noabbrev]{cleveref}
\expandafter\def\csname ver@etex.sty\endcsname{3000/12/31}

\usepackage{autonum}                         

\newtheorem{theorem}{Theorem}[section]
\newtheorem{Mtheorem}{Main Theorem}
\newtheorem*{acknowledgment}{Acknowledgment}

\newtheorem{lemma}[theorem]{Lemma}

\newtheorem{remark}[theorem]{Remark}

\crefname{theorem}{Theorem}{Theorems}						
\creflabelformat{theorem}{#2{#1}#3}							
\crefname{Mtheorem}{Main Theorem}{Main Theorems}			
\creflabelformat{Mtheorem}{#2{#1}#3}						
\crefname{lemma}{Lemma}{Lemmata}							
\creflabelformat{lemma}{#2{#1}#3}							
\crefname{corollary}{Corollary}{Corollaries}				
\creflabelformat{corollary}{#2{#1}#3}						
\crefname{proposition}{Proposition}{Propositions}			
\creflabelformat{proposition}{#2{#1}#3}						
\crefname{ineq}{inequality}{inequalities}					
\creflabelformat{ineq}{#2{\upshape(#1)}#3}					
\crefname{cond}{condition}{conditions}						
\creflabelformat{cond}{#2{\upshape(#1)}#3}					
\crefname{hypoth}{Hypothesis}{Hypotheses}					
\creflabelformat{hypoth}{#2{#1}#3}							
\crefname{definition}{Definition}{Definitions}				
\creflabelformat{def}{#2{#1}#3}								
\crefname{appsec}{Appendix}{Appendices}

\allowdisplaybreaks[1]

\let\originalleft\left
\let\originalright\right
\renewcommand{\left}{\mathopen{}\mathclose\bgroup\originalleft}
\renewcommand{\right}{\aftergroup\egroup\originalright}

\def\({\mathopen{}\left(}
\def\){\right)\mathclose{}}

\makeatletter
\renewcommand*{\eqref}[1]{\hyperref[{#1}]{\textup{\tagform@{\ref*{#1}}}}}
\makeatother

\newcommand*{\eqdef}{\mathrel{\vcenter{\baselineskip0.5ex \lineskiplimit0pt\hbox{.}\hbox{.}}}=}

\def\id{\mathbbm{1}}
\def\cx{\mathbbm{C}}

\def\rl{\mathbbm{R}}
\def\N{\mathbbm{N}}
\def\P{\mathbbm{P}}

\def\cA{\mathcal{A}}

\def\cC{\mathcal{C}}

\def\cE{\mathcal{E}}
\def\cF{\mathcal{F}}

\def\cH{\mathcal{H}}

\def\cL{\mathcal{L}}

\def\cX{\mathcal{X}}

\def\Im{\mathrm{Im}}

\def\Ran{\mathrm{Ran}}
\def\Re{\mathrm{Re}}

\def\Spec{\mathrm{Spec}}

\def\vol{\mathrm{vol}}
\def\Vol{\mathrm{Vol}}

\def\Conn{\mathrm{Conn}}
\def\dR{\mathrm{dR}}
\def\del{\partial}
\def\delbar{\overline{\partial}}
\def\rd{\mathrm{d}}
\def\rG{\mathrm{G}}

\title[On the bifurcation theory of the Ginzburg--Landau equations]{On the bifurcation theory of the Ginzburg--Landau equations}
\date{\today}
\keywords{Ginzburg--Landau equations, nonminimal solutions, bifurcation theory}
\subjclass[2020]{35Q56, 53C07, 58E15, 58J55}

\author{\'Akos Nagy}
\address[\'Akos Nagy]{University of California, Santa Barbara}
\urladdr{\href{https://akosnagy.com}{akosnagy.com}}
\email{\href{mailto:contact@akosnagy.com}{contact@akosnagy.com}}

\author{Gon\c{c}alo Oliveira}
\address[Gon\c{c}alo Oliveira]{IST Austria}
\urladdr{\href{https://sites.google.com/view/goncalo-oliveira-math-webpage/home}{sites.google.com/view/goncalo-oliveira-math-webpage/home}}
\email{\href{mailto:galato97@gmail.com}{galato97@gmail.com}}

\calclayout
\hypersetup{
	pdffitwindow	= true,
	unicode			= true,
	pdffitwindow	= true,
	pdftoolbar		= false,
	pdfmenubar		= false,
	pdfstartview	= {FitH},
	pdfauthor		= {\'Akos Nagy, Gon\c{c}alo Oliveira},
	colorlinks		= true,
	linkcolor		= black,
	citecolor		= black,
	filecolor		= black,
	urlcolor		= blue
}

\begin{document}

\begin{abstract}
	We construct nonminimal and irreducible solutions to the Ginzburg--Landau equations on closed manifolds of arbitrary dimension with trivial first real cohomology. Our method uses bifurcation theory where the ``bifurcation points'' are characterized by the eigenvalues of a Laplace-type operator. To our knowledge these are the first such examples on nontrivial line bundles.
\end{abstract}

\maketitle

\section{Introduction}

We consider Ginzburg--Landau theory, also known as abelian Yang--Mills--Higgs theory, on Riemannian manifolds. More concretely, if $\( X, g \)$ is an $N$-dimensional, closed, oriented, Riemannian manifold, $\( \cL, h \)$ is a Hermitian line bundle over $X$, and $\tau, \kappa$ are two positive coupling constants, then the \emph{Ginzburg--Landau (free) energy} of a unitary connection $\nabla$ and smooth section $\upphi$ (both on $\cL$) is given by
\begin{equation}
	\( \nabla, \upphi \) \mapsto \frac{1}{2} \int\limits_X \( \left| F_\nabla \right|^2 + \left| \nabla \upphi \right|^2 + \tfrac{\kappa^2}{2} \( \tau - |\upphi|^2 \)^2 \) \ \vol_g. \label{eq:gle}
\end{equation}
By \cite{PS20}*{Proposition~A.1} every critical point of \eqref{eq:gle} is gauge equivalent to a smooth one, which in turn is a solution to
\begin{subequations}
	\begin{align}
		\rd^* F_\nabla - i \: \Im \( h \( \nabla \upphi, \upphi \) \)		&= 0, \label{eq:gl1} \\
		\nabla^* \nabla \upphi - \kappa^2 \( \tau - |\upphi|^2 \) \upphi	&= 0. \label{eq:gl2}
	\end{align}
\end{subequations}
We call \cref{eq:gl1,eq:gl2} the \emph{Ginzburg--Landau equations}. These are nonlinear, second order, elliptic partial differential equations which are invariant by the action of the group of automorphisms of $\( \cL, h \)$, also known as the gauge group. If a unitary connection $\nabla^0$ satisfies the abelian Yang--Mills equation (also known as the source-free Maxwell's equation)
\begin{equation}
	\rd^* F_{\nabla^0} = 0, \label{eq:normal}
\end{equation}
then the pair $\( \nabla^0, 0 \)$ solves the Ginzburg--Landau equations; such a pair is said to be a {\em normal phase solution}. Notice that \cref{eq:normal} is independent of $\tau$ and $\kappa$. As is common in abelian gauge theories, we call a pair $\( \nabla, \upphi \)$ {\em reducible} if $\upphi$ vanishes identically, and {\em irreducible} otherwise. A solution to the Ginzburg--Landau equations is reducible if and only if it is a normal phase solution.

\smallskip

On closed manifolds absolute minimizers of the Ginzburg--Landau energy \eqref{eq:gle} exists (cf; \cites{N17b,PS20}) which are automatically solutions Ginzburg--Landau \cref{eq:gl1,eq:gl2}. The minimizers, often called vortices, are well-understood, especially on K\"ahler manifolds and for critical coupling, that is, when $\kappa = \tfrac{1}{\sqrt{2}}$; cf. \cites{JT80,B90,GP93} and more recently \cites{CERS17,N17b}. The critically coupled case has special properties that the others lack, for example ``self-duality'' via a Bogomolny-type trick.

\smallskip

Much less is known about nonminimal solutions. In \cite{PS20}, Pigati and Stern constructed irreducible solutions on (topologically) trivial line bundles over closed Riemannian manifolds. As these solutions have positive energy, they cannot be the absolute minima of Ginzburg--Landau energy \eqref{eq:gle}. Furthermore, in our companion paper \cite{nagy_nonminimal_2022}, we constructed nonminimal critical points on nontrivial line bundles over oriented and closed surfaces.

\smallskip

In this paper, motivated by and building on works of \cites{Parker1992a,Parker1992b,Parker1993,N17b,CERS17}, we construct new, nonminimal, and irreducible solutions to the Ginzburg--Landau \cref{eq:gl1,eq:gl2}. To the best of our knowledge, on nontrivial bundles and in dimension greater than two, these solutions are the only known nonminimal and irreducible solutions so far, and together with the solutions of \cites{PS20,nagy_nonminimal_2022}, these solutions are the only known nonminimal and irreducible solutions on any line bundle.

\medskip

\subsection*{Summary of main results}

We construct solutions on closed manifolds satisfying certain topological/geometric conditions. The proof uses a technique inspired by Lyapunov--Schmidt reduction; cf. \cite{GW13}*{Chapter~5}.

\begin{Mtheorem}
	\label{Mtheorem:branch}
	Let $X$ be a closed, oriented, $N$-dimensional, Riemannian manifold with a Hermitian line bundle $\cL$, and $\( \nabla^0, 0 \)$ be a normal phase solution on $\cL$. Let $\Delta_0 \eqdef \( \nabla^0 \)^* \nabla^0$ acting on square integrable sections of $\cL$ and $\lambda \in \Spec \( \Delta_0 \)$. Assume $X$ has trivial first de Rham cohomology.

	Then there exists $t_0 > 0$ and for each $t \in \( 0, t_0 \)$ an element $\Phi_t \in \ker \( \Delta_0 - \lambda \id \)$ with unit $L^2$-norm such that there is a (possibly discontinuous) branch of triples	
	\begin{equation}
		\left\{ \ \( A_t, \upphi_t, \tau_t \) \in \Omega^1 \times \( L_1^2 \cap L^N \) \times \rl_+ \ | \ t \in \( 0, t_0 \) \ \right\}, \label{eq:branch}
	\end{equation}
	of the form
	\begin{equation}
		\( A_t, \upphi_t, \tau_t \) = \( t^2 \cA_t, t \Phi_t + t^3 \Psi_t, \tfrac{\lambda}{\kappa^2} + t^2 \epsilon_t \),
	\end{equation}
	such that the family
	\begin{equation}
		\left\{ \ \( \cA_t, \Psi_t, \epsilon_t \) \ | \ t \in \( 0, t_0 \) \ \right\},
	\end{equation}
	is determined by $\Phi_t$ and is bounded in $L_1^2 \times \( L_1^2 \cap L^N \cap \( \ker \( \Delta_0 - \lambda \id \)^\perp \) \) \times \rl_+$, and for each $t \in \( 0, t_0 \)$ the pair $\( \nabla^0 + A_t, \upphi_t \)$ is an irreducible solution to the Ginzburg--Landau \cref{eq:gl1,eq:gl2} with $\tau_t$.
\end{Mtheorem}

\begin{remark}
	\label{remark:branch_2}
	In \Cref{theorem:branch_2} we consider a similar case, where we get a weaker result. Namely, we remove the assumption that $X$ has trivial first de Rham cohomology, and replace it with the conditions that $\kappa^2 \geqslant \tfrac{1}{2}$, $X$ is K\"ahler, $\nabla^0$ is Hermitian Yang--Mills, and $\cL$ carries nontrivial holomorphic sections with respect to $\nabla^0$, and $\lambda = \min \( \Spec \( \Delta_0 \) \)$. This result is a generalization of the main result of \cite{CERS17}, which only covered closed surfaces of high genus with line bundles of high degree. Our result extends this to all closed K\"ahler manifolds and line bundles.
\end{remark}

\medskip

\subsection*{Organization of the paper} In \Cref{sec:GL_theory}, we give a brief introduction to the important geometric analytic aspects of the Ginzburg--Landau theory that are needed to prove our results. In \Cref{sec:bifurcation_stuff}, we use bifurcation theory to construct novel solutions to the Ginzburg--Landau equations, in any dimensions, on closed Riemannian manifolds with zero first Betti number. In \Cref{ss:proof_of_branch_2}, we prove a similar result for K\"ahler manifolds.

\begin{acknowledgment}
	The first author is thankful to Israel Michael Sigal, Steve Rayan, and Daniel Stern for their help with various parts of the paper.
	
	The second author was supported by Funda\c{c}\~ao Serrapilheira 1812-27395, by CNPq grants 428959/2018-0 and 307475/2018-2, and FAPERJ through the program Jovem Cientista do Nosso Estado E-26/202.793/2019.

	The authors are thankful to the anonymous Referee for their thorough feedback.
\end{acknowledgment}

\bigskip

\section{Ginzburg--Landau theory on manifolds}
\label{sec:GL_theory}

Let $\( X, g \)$ be a closed, oriented, Riemannian manifold of dimension $N$. Let us fix a connection $\nabla^0$ that satisfies \cref{eq:normal}. We define the Sobolev norms of $\cL$-valued forms via the Levi-Civita connection of $\( X, g \)$ and the connection $\nabla^0$. Note that the induced topologies are the same for any choice of $\nabla^0$ and since the moduli of normal phase solutions is compact (modulo gauge), and if a Coulomb-type gauge fixing condition is chosen (with respect to a reference connection), then the family of norms are in fact uniformly equivalent, that is, for all $k$ and $p$, there exists a number $C_{k, p} \geqslant 1$, such that for any two Sobolev $L_k^p$-norms, $\| \cdot \|_{L_k^p}$ and $\| \cdot \|_{L_k^p}^\prime$, given by two connections satisfying \cref{eq:normal} and the Coulomb condition, we have
\begin{equation}
	\frac{1}{C_{k, p}} \| \cdot \|_{L_k^p}^\prime \leqslant \| \cdot \|_{L_k^p} \leqslant C_{k, p} \| \cdot \|_{L_k^p}^\prime.
\end{equation}
For any Hermitian bundle $E$ that is constructed from $\cL$ and tensor bundles over $X$, let $L_k^p \( E \)$ be the $L_k^p$-completion of the space of smooth sections on $E$ with respect $\nabla^0$ and the Levi-Civita connection on $X$.

\smallskip

Let $\Conn_\cL$ be the $L_1^2$-completion of the space of smooth unitary connections on $\cL$ that are in Coulomb gauge with respect to $\nabla^0$. Then $\Conn_\cL$ is an affine space over
\begin{equation}
	i \Omega_{\rd^*}^1 \eqdef \ker \( \rd^* : L_1^2 \( i T^* X \) \rightarrow L^2 \( i \underline{\rl} \) \),
\end{equation}
and thus the tangent bundle, $T \Conn_\cL$, is canonically isomorphic to $\Conn_\cL \times i \Omega_{\rd^*}^1$. Let $\Pi_{\rd^*}$ be the $L^2$-orthogonal projection onto the space of smooth and coclosed $1$-forms. More concretely, let $G : L^2 \rightarrow L_2^2$ be the Green's operator on $X$. Then for each $a \in L_1^2 \( i T^* X \)$, let
\begin{equation}
	\Pi_{\rd^*} (a) \eqdef a - \rd \( G \( \rd^*a \) \) \in i \Omega_{\rd^*}^1.
\end{equation}
Then $\Pi_{\rd^*}^2 = \Pi_{\rd^*}$, $\rd^* \circ \Pi_{\rd^*} = 0$, and if $a$ is coclosed, then $\Pi_{\rd^*} (a) = a$. In particular, $\Pi_{\rd*} \circ \rd^* = \rd^*$.

Let $p > N$ any and $\cX$ be the completion of smooth sections of $\cL$ via the norm
\begin{equation}
	\| \upphi \|_\cX \eqdef \| \nabla ^0 \upphi \|_{L^2} + \| \upphi \|_{L^p}. \label{eq:X_norm_def}
\end{equation}
Note that $\cX \subset L^4$. Now let us define $\cC_\cL \eqdef \Conn_\cL \times \cX$. For each $\( \nabla, \upphi \) \in \cC_\cL$, let
\begin{equation}
	j \( \nabla, \upphi \) \eqdef i \: \Pi_{\rd^*} \( \Im \( h \( \nabla \upphi, \upphi \) \) \).
\end{equation}
By \cref{eq:X_norm_def}, for each $\( \nabla, \upphi \) \in \cX$, we have that
\begin{equation}
	j \( \nabla, \upphi \) \in L^{\frac{2 N}{N + 2}} \( i T^* X \) \subset \( L_1^2 \( i T^* X \) \)^*. \label{eq:j_in_dual}
\end{equation}
Furthermore, for all smooth, imaginary-valued $1$-form, $b$, we have
\begin{align}
	g \( \Im \( h \( \nabla^0 \upphi, \upphi \) \) \middle| - i b \)	&= g \( \tfrac{h \( \nabla^0 \upphi, \upphi \) - h \( \upphi, \nabla^0 \upphi \)}{2i}, - i b \) \\
		&= \tfrac{1}{2} \( \( g \otimes h \) \( \nabla^0 \upphi, \tfrac{1}{i} \( - i b \) \upphi \) + \( g \otimes h \) \( \tfrac{- 1}{i} \( - i b \) \upphi, \nabla^0 \upphi \) \) \\
		&= - \Re \( \( g \otimes h \) \( \nabla^0 \upphi, b \upphi \) \),
\end{align}
thus, if $b$ is also coclosed, then integrating both sides yield
\begin{equation}
	\left\langle j \( \nabla^0, \upphi \) \middle| b \right\rangle = - \Re \( \left\langle \nabla^0 \upphi \middle| b \upphi \right\rangle \), \label{eq:j-a_pairing}
\end{equation}
and by continuity, \cref{eq:j-a_pairing} holds for all $b \in L^{\frac{2 N}{N - 2}} \left( i T^* X \right)$.

For each $\( \nabla, \upphi \) \in \cC_\cL$ and $\( b, \uppsi \) \in i \Omega_{\rd^*}^1 \times \cX \cong T_{\( \nabla, \upphi \)} \cC_\cL$, let us define the \emph{Ginzburg--Landau gradient map} to be
\begin{equation}
	\cF_{\tau, \kappa} \( \nabla, \upphi \) \( b, \uppsi \) = \left\langle F_\nabla \middle| \rd b \right\rangle - \left\langle j \( \nabla, \upphi \) \middle| b \right\rangle + \left\langle \nabla \upphi \middle| \nabla \uppsi \right\rangle + \kappa^2 \left\langle \( |\upphi|^2 - \tau \) \upphi \middle| \uppsi \right\rangle. \label{eq:GL_gradient}
\end{equation}

\smallskip

\begin{lemma}
	The Ginzburg--Landau gradient map \eqref{eq:GL_gradient} is an analytic section of the cotangent bundle of $\cC_\cL$.
\end{lemma}

\begin{proof}
	Let us write $a \eqdef \nabla - \nabla^0 \in L_1^2$. Then, using \cref{eq:j-a_pairing}, we can rewrite \cref{eq:GL_gradient} as
	\begin{align}
		\cF_{\tau, \kappa} \( \nabla^0 + a, \upphi \) \( b, \uppsi \)	&= \underbrace{\left\langle F_{\nabla^0} \middle| \rd b \right\rangle}_{= 0} + \left\langle \rd a \middle| \rd b \right\rangle - \underbrace{\left\langle j \( \nabla^0, \upphi \) \middle| b \right\rangle}_{= - \Re \( \left\langle \nabla^0 \upphi \middle| b \upphi \right\rangle \)} + \left\langle a \upphi \middle| b \upphi \right\rangle \\
			& \quad + \left\langle \nabla^0 \upphi \middle| \nabla^0 \uppsi \right\rangle + \Re \( \left\langle \nabla^0 \upphi \middle| a \uppsi \right\rangle \) + \Re \( \left\langle \nabla^0 \uppsi \middle| a \upphi \right\rangle \) + \left\langle a \upphi \middle| a \uppsi \right\rangle \\
			& \quad + \kappa^2 \left\langle |\upphi|^2 \upphi \middle| \uppsi \right\rangle - \kappa^2 \tau \left\langle \upphi \middle| \uppsi \right\rangle,
	\end{align}
	and after rearrangement
	\begin{align}
		\cF_{\tau, \kappa} \( \nabla^0 + a, \upphi \) \( b, \uppsi \)	&= \left\langle \rd a \middle| \rd b \right\rangle + \left\langle \nabla^0 \upphi \middle| \nabla^0 \uppsi \right\rangle - \kappa^2 \tau \left\langle \upphi \middle| \uppsi \right\rangle \\
			& \quad + \Re \( \left\langle \nabla^0 \upphi \middle| b \upphi \right\rangle \) + \Re \( \left\langle \nabla^0 \upphi \middle| a \uppsi \right\rangle \) + \Re \( \left\langle \nabla^0 \uppsi \middle| a \upphi \right\rangle \) \\
			& \quad + \left\langle a \upphi \middle| b \upphi \right\rangle + \left\langle a \upphi \middle| a \uppsi \right\rangle + \kappa^2 \left\langle |\upphi|^2 \upphi \middle| \uppsi \right\rangle.
	\end{align}
	Now every term above is polynomial, so $\cF_{\tau, \kappa}$ is analytic exactly if it is continuous. The terms in the first row are immediately continuous by the definition of the norms on $\cC_\cL$. When $N = 2$, the remaining terms are clearly continuous, while if $N > 2$, then they can be bounded by the appropriate norms, using H\"older's inequality with power $p = N$ on $\upphi$ and $\uppsi$ and $q = \tfrac{2 N}{N - 2}$ on the $1$-forms, and the using the embedding $L_1^2 \subset L^{\frac{2 N}{N - 2}}$.
\end{proof}

Before the next lemma, note that $\( \nabla, \upphi \) \in \cC_\cL$ is a weak solution to the Ginzburg--Landau \cref{eq:gl1,eq:gl2}, exactly when for all pairs, $\( a, \uppsi \) \in i \Omega_{\rd^*}^1 \oplus \cX$, we have that
\begin{subequations}
\begin{align}
	\langle F_\nabla | \rd a \rangle + \langle i \: \Im \( h \( \upphi, \nabla \upphi \) \) | a \rangle &= 0, \label{eq:wGL1} \\
	\langle \nabla \upphi | \nabla \uppsi \rangle - \kappa^2 \left\langle \( \tau - |\upphi|^2 \) \upphi \middle| \uppsi \right\rangle	&= 0. \label{eq:wGL2}
\end{align}
\end{subequations}

\smallskip

\begin{lemma}
\label{lemma:projected_eqs}
	A pair $\( \nabla, \upphi \) \in \cC_\cL$ is a solution to \cref{eq:wGL1,eq:wGL2} exactly when
	\begin{equation}
		\cF_{\tau, \kappa} \( \nabla, \upphi \) = 0. \label{eq:GL_grad}
	\end{equation}
\end{lemma}

\begin{proof}
	First of all, \cref{eq:wGL1,eq:wGL2} a fortiori imply \cref{eq:GL_grad}. It is easy to see that \cref{eq:wGL2} is equivalent to $\cF_{\tau, \kappa} \( \nabla, \upphi \) (0, \uppsi) = 0$, for all smooth $\uppsi$. Thus, we only need to prove that if \cref{eq:GL_grad} holds, then \cref{eq:wGL1} also holds.

	Let us assume that \cref{eq:GL_grad} holds, thus, in particular \cref{eq:wGL2} also holds, and prove \cref{eq:gl1}. Let $f$ be a smooth, real-valued function on $X$ and compute
	\begin{align}
		\left\langle i \Im \( h \( \upphi, \nabla \upphi \) \) \middle| i \rd f \right\rangle	&= \tfrac{1}{2} \( \left\langle h \( \upphi, \nabla \upphi \) \middle| i \rd f \right\rangle - \left\langle h \( \nabla \upphi, \upphi \) \middle| i \rd f \right\rangle \) \\
			&= - \tfrac{1}{2} \( \left\langle \( i \rd f \) \upphi \middle| \nabla \upphi \right\rangle + \left\langle \nabla \upphi \middle| \( i \rd f \) \upphi \right\rangle \) \\
			&= - \Re \( \left\langle \nabla \upphi \middle| \( \rd \( i f \) \) \upphi \right\rangle \) \\
			&= - \Re \( \left\langle \nabla \upphi \middle| \( \nabla \( i f \upphi \) - i f \nabla \upphi \) \right\rangle \) \\
			&= \Re \( i f \( \| \nabla \upphi \|_{L^2}^2 - \kappa^2 \| \upphi \|_{L^2}^2 + \kappa^2 \| \upphi \|_{L^4}^4 \) \) \\
			&= 0,
	\end{align}
	where to get the fifth line we used \cref{eq:wGL2} with $\uppsi = f \upphi$. Hence $i \: \Im \( h \( \nabla \upphi, \upphi \) \)$ is $L^2$-orthogonal to $\Ran \( \rd : L_1^2 \( i \underline{\rl} \) \rightarrow L^2 \( i T^* X \) \)$ and thus, by Hodge's Theorem, $i \: \Im \( h \( \nabla \upphi, \upphi \) \) \in \Omega_{\rd^*}^1$. In particular, $i \: \Im \( h \( \nabla \upphi, \upphi \) \) = i \: \Pi_{\rd^*} \( \Im \( h \( \nabla \upphi, \upphi \) \) \) = j \( \nabla, \upphi \)$, and so
	\begin{equation}
		\rd^* F_\nabla - i \: \( \Im \( h \( \nabla \upphi, \upphi \) \) \) = \rd^* F_\nabla - j \( \nabla, \upphi \) = 0.
	\end{equation}
	which concludes the proof.
\end{proof}

\begin{remark}
	The same argument shows that for any real-valued function $w$, if $\nabla^* \nabla \upphi = w \upphi$ holds, then
	\begin{equation}
		\rd^* F_\nabla - i \: \Im \( h \( \nabla \upphi, \upphi \) \) = 0 \quad \Leftrightarrow \quad \rd^* F_\nabla - j \( \nabla, \upphi \) = 0.
	\end{equation}
\end{remark}

\bigskip

\section{Solutions in higher dimensions through bifurcations}
\label{sec:bifurcation_stuff}

In this section we allow the underlying compact manifold to be higher dimensional and prove \Cref{Mtheorem:branch}, using bifurcation theory.

\medskip

\begin{proof}[Proof of \Cref{Mtheorem:branch}:]
	Using \cite{W04}*{Theorem~5.1~part~(ii) with $(k, p) = (1, 2)$}, we get that we can use the norm $\| A \|_{i \Omega_{\rd^*}^1} \eqdef \| \rd A \|_{L^2}$, to define the Hilbert structure on $i \Omega_{\rd^*}^1$. Furthermore, by \Cref{lemma:projected_eqs}, it is enough to solve the ``projected'' \cref{eq:GL_grad}.

	First we show that $\nabla$ can be eliminated from \cref{eq:GL_grad} as follows: Fix $\upphi \in \cX$ and consider the (weak version of the) following equation for $A \in i \Omega_{\rd^*}^1$:
	\begin{equation}
		\rd^* \rd A + \Pi_{\rd^*} \( |\upphi|^2 A \) = j \( \nabla^0, \upphi \). \label{eq:gl1_rewrite}
	\end{equation}
	It is easy to see that $A$ solves \cref{eq:gl1_rewrite} exactly when
	\begin{equation}
		\forall b \in i \Omega_{\rd^*}^1 : \quad \cF_{\tau, \kappa} \( \nabla^0 + A, \upphi \) (b, 0) = 0,
	\end{equation}
	Let us define the Hilbert space
	\begin{equation}
		\cH_\upphi \eqdef \overline{\left\{ \: A \in i \Omega_{\rd^*}^1 \: \middle| \: \left\| A \right\|_\upphi^2 \eqdef \left\| \rd A \right\|_{L^2}^2 + \left\| \upphi A \right\|_{L^2}^2 < \infty \: \right\}}^{\left\| \cdot \right\|_\upphi},
	\end{equation}
	It is easy to verify that $\cH_\upphi \subseteq i \Omega_{\rd^*}^1$, and thus the closure is unnecessary. Furthermore, the map
	\begin{equation}
		A \mapsto \left\langle j \( \nabla^0, \upphi \) \middle| A \right\rangle,
	\end{equation}
	is in $\cH_\upphi^*$. Thus, by the Riesz Representation Theorem, there exists a unique solution to \cref{eq:gl1_rewrite} in $\cH_\upphi$. Let $\cA \( \upphi \)$ be this unique solution. Straightforward computation shows that $\cA$ is a continuous (and compact) map from $\cX$ to $i \Omega_{\rd^*}^1$. Using \cref{eq:gl1_rewrite,eq:j-a_pairing} can also get that
	\begin{equation}
		\| \cA \( \upphi \) \|_{i \Omega_{\rd^*}^1} \leqslant \tfrac{1}{2} \| \upphi \|_\cX^2.
	\end{equation}
	Using again \cite{W04}*{Theorem~5.2} and \cref{eq:j_in_dual}, we can bootstrap to get
	\begin{equation}
		\| \cA \( \upphi \) \|_{L_2^{\frac{2 N}{N + 2}}} = O \( \| \upphi \|_\cX^2 \). \label[ineq]{ineq:A_bound}
	\end{equation}

	Let us omit the connection and $\kappa$ from the notation, and define
	\begin{equation}
		\forall \upxi \in \cX : \quad \cF_\tau \( \upphi \) \( \upxi \) \eqdef \left\langle \( \nabla^0 + \cA \( \upphi \) \) \upphi \middle| \( \nabla^0 + \cA \( \upphi \) \) \upxi \right\rangle - \kappa^2 \left\langle \( \tau - |\upphi|^2 \) \upphi \middle| \upxi \right\rangle.
	\end{equation}
	Thus solving \cref{eq:GL_grad} is now reduced to solving $\cF_\tau \( \upphi \) = 0$. We regard $\cF_\tau$ as a section of the cotangent bundle of $\cX$. Moreover the function $\( \tau, \upphi \) \mapsto \cF_\tau \( \upphi \)$ is smooth and $\cF_\tau (0) = 0$.

	If $\lambda \in \Spec \( \Delta_0 \)$ and $\Phi$ is an $L_1^2$-section of $\cL$ such that $\Delta_0 \Phi = \lambda \Phi$ (weakly), then $\Phi \in \cX$, by elliptic regularity, in fact, it is smooth. Let $\tau_0 \eqdef \tfrac{\lambda}{\kappa^2}$. Then, a simple computation shows that $\( D \cF_{\tau_0} \) (0) \( \Phi \) = 0$. Assume that $\| \Phi \|_{L^2} = 1$. We next show that there are $t_0, \epsilon_0 \in \rl_+$, such that if $t \in \( 0, t_0 \)$ and $\epsilon \in \( - \epsilon_0, \epsilon_0 \)$, then there exists $\Psi = \Psi \( \Phi, t, \epsilon \) \in \cX$ such that $\Psi$ is $L^2$-orthogonal to $\ker \( \Delta_0 - \lambda \id \)$ and
	\begin{equation}
		\forall \upxi \in \cX : \quad \upxi \perp_{L^2} \ker \( \Delta_0 - \lambda \id \) \: \Rightarrow \: \cF_{\tau_0 + \epsilon} \( t \( \Phi + \Psi \( \Phi, t, \epsilon \) \) \) \( \upxi \) = 0. \label{eq:Orthogonal component of the equation}
	\end{equation}
	In order to find $\Psi \( \Phi, t, \epsilon \)$, we first solve another equation: let $G_\lambda : L_k^2 \rightarrow L_{k + 2}^2$ be the Green's operator corresponding to $\Delta_0 - \lambda \id$, acting as the zero operator on $\ker \( \Delta_0 - \lambda \id \)$. Note that $G_\lambda$ is compact as a map from $L_k^2$ to itself. Now consider the fixed-point equation for an $L_k^2$-section of $\cL$, $\widetilde{\Psi}$, with $k = \left\lceil \tfrac{n}{2} + 3 \right\rceil$:
	\begin{equation}
		\widetilde{\Psi} = \overbrace{\rG_\lambda \( \epsilon \widetilde{\Psi} - 2 \cA \( t \Phi + \widetilde{\Psi} \)^* \( \nabla^0 \( t \Phi + \widetilde{\Psi} \) \) - \left| \cA \( t \Phi + \widetilde{\Psi} \) \right|^2 \( t \Phi + \widetilde{\Psi} \) - \kappa^2 \left| t \Phi + \widetilde{\Psi} \right|^2 \( t \Phi + \widetilde{\Psi} \) \)}^{\mathbb{G} \( \widetilde{\Psi} \) \eqdef}. \label{eq:FixPt}
	\end{equation}
	We prove two statements about \cref{eq:FixPt}. First, using \cref{ineq:A_bound}, we can show that \cref{eq:FixPt} satisfies the conditions of the Schauder fixed-point theorem near the origin of $L_k^2$ if $\( t, \epsilon \) \in \( 0, t_0 \) \times \( - \epsilon_0, \epsilon_0 \)$, for small enough $t_0, \epsilon_0 \in \rl_+$, independent of $\Phi$. Now let $\Psi \( \Phi, t, \epsilon \) = t^{- 1} \widetilde{\Psi} \in L_k^2 \subset \cX$ be the unique solution close to the origin given by \cref{eq:FixPt}. This is because, using using various embeddings of the form $L_k^2 \hookrightarrow L^p$, we get that $\mathbb{G} : L_k^2 \rightarrow L_k^2$ satisfies
	\begin{equation}
		\| \mathbb{G} \( \widetilde{\Psi} \) \|_{L_k^2} = O \( \epsilon \| \widetilde{\Psi} \|_{L_k^2} + \| t \Phi + \widetilde{\Psi} \|_{L_k^2}^2 + \| t \Phi + \widetilde{\Psi} \|_{L_k^2}^5 + \| t \Phi + \widetilde{\Psi} \|_{L_k}^3 \).
	\end{equation}
	Since the image of a bounded set in $L_k^2$ by $\mathbb{G}$ is compact in $L_k^2$, we have that for $t, \epsilon, R \in \rl_+$ small enough, the image of the closed $R$-ball in $L_k^2$ is compactly contained in itself, thus the Schauder fixed-point theorem can be applied, to get a unique solution in that ball, which we call $\widetilde{\Psi} \( \Phi, t, \epsilon \)$.

	Second, we prove that $\Psi \( \Phi, t, \epsilon \) \eqdef t^{- 1} \widetilde{\Psi} \( \Phi, t, \epsilon \)$ solve \cref{eq:Orthogonal component of the equation}. This is because $\Psi \( \Phi, t, \epsilon \)$, in fact and by construction, satisfies
	\begin{equation}
		\nabla^* \nabla \( t \( \Phi + \Psi \( \Phi, t, \epsilon \) \) \) - \kappa^2 \( \tfrac{\lambda}{\kappa^2} + \epsilon - \left| t \( \Phi + \Psi \( \Phi, t, \epsilon \) \) \right|^2 \) \in \( \ker \( \Delta_0 - \lambda \id \) \)^{\perp_{L^2}}.
	\end{equation}
	with $\nabla = \nabla^0 + \cA \( t \( \Phi + \Psi \( \Phi, t, \epsilon \) \) \)$.

	Furthermore, it is easy to see that $\Psi = \Psi \( \Phi, t, \epsilon \) \in \cX$ depends continuously on $\( \Phi, t, \epsilon \)$ and
	\begin{equation}
		\| \Psi \|_\cX = O \( t^2 \), \label[ineq]{ineq:Psi_bound}
	\end{equation}
	where the hidden constant on the right-hand side is independent of $\( \Phi, t, \epsilon \)$ (once $t$ and $|\epsilon|$ are small enough), and thus \cref{ineq:A_bound} with the (re)definition $\cA \( \Phi, t \) \eqdef \cA \( t \( \Phi + \Psi \) \)$ becomes
	\begin{equation}
		\| \cA \( \Phi, t \) \|_{i \Omega_{\rd^*}^1} = O \( t^2 \). \label[ineq]{ineq:new_A_bound}
	\end{equation}
	Now for any $\Phi$ as above and $\( t, \epsilon \) \in \( 0, t_0 \) \times \( - \epsilon_0, \epsilon_0 \)$, the equation $\cF_{\tau_0 + \epsilon} \( t \( \Phi + \Psi \( \Phi, t, \epsilon \) \) \) = 0$ is equivalent to
	\begin{equation}
		\cF_{\tau_0 + \epsilon} \( t \( \Phi + \Psi \( \Phi, t, \epsilon \) \) \)|_{\ker \( \Delta_0 - \lambda \id \)} = 0.
	\end{equation}
	Let again $\Psi_t = \Psi \( \Phi, t, \epsilon \)$ and $A_t = \cA \( \Phi, t \)$. Now evaluate the following
	\begin{align}
		t^{- 2} \cF_{\tau_0 + \epsilon} \( t \( \Phi + \Psi_t \) \) \( \Phi \)	&= t^{- 2} \cF_{\tau_0 + \epsilon} \( t \( \Phi + \Psi_t \) \) \( \Phi + \Psi \) - t^{- 2} \underbrace{\cF_{\tau_0 + \epsilon} \( t \( \Phi + \Psi \) \) \( \Psi \)}_{= 0} \\
			&= \left\| \( \nabla^0 + A_t \) \( \Phi + \Psi_t \) \right\|_{L^2}^2 - \( \lambda + \kappa^2 \epsilon \) \left\| \Phi + \Psi_t \right\|_{L^2}^2 + t^2 \kappa^2 \left\| \Phi + \Psi_t \right\|_{L^4}^4,
	\end{align}
	thus $\cF_{\tau_0 + \epsilon} \( t \( \Phi + \Psi_t \) \) \( \Phi \) = 0$ can be solved for $\epsilon$ and get
	\begin{equation}
		\epsilon \( \Phi, t \) = \frac{\left\| \( \nabla^0 + A_t \) \( \Phi + \Psi_t \) \right\|_{L^2}^2 - \lambda \left\| \Phi + \Psi_t \right\|_{L^2}^2 + t^2 \kappa^2 \left\| \Phi + \Psi_t \right\|_{L^4}^4}{\kappa^2 \left\| \Phi + \Psi_t \right\|_{L^2}^2}.
	\end{equation}
	Now using \cref{ineq:Psi_bound,ineq:new_A_bound}, we get that
	\begin{equation}
		\left| \epsilon \( \Phi, t \) \right| = O \( t^2 \),
	\end{equation}
	thus, after possibly shrinking $t_0 > 0$, we get that solving $\cF_{\tau_0 + \epsilon} \( t \( \Phi + \Psi \( \Phi, t, \epsilon \) \) \) = 0$ is further reduced to
	\begin{equation}
		\forall \Phi^\prime \in \ker \( \Delta_0 - \lambda \id \) \cap \( \cx \Phi \)^{\perp_{L^2}} : \quad \cF_{\tau_0 + \epsilon \( \Phi, t \)} \( t \( \Phi + \Psi \( \Phi, t, \epsilon \( \Phi, t \) \) \) \) \( \Phi^\prime \) = 0.
	\end{equation}
	Note that $\ker \( \Delta_0 - \lambda \id \) \cap \( \cx \Phi \)^{\perp_{L^2}}$ and the tangent space of $\P \( \ker \( \Delta_0 - \lambda \id \) \)$ at $\cx \Phi$ are canonically isomorphic. Thus, for each $t \in \( 0, t_0 \)$, the function
	\begin{equation}
		\Upsilon_t \( \cx \Phi \) \eqdef \cF_{\tau_0 + \epsilon \( \Phi, t \)} \( t \( \Phi + \Psi \( \Phi, t, \epsilon \( \Phi, t \) \) \) \),
	\end{equation}
	can be viewed as a continuous $1$-form on a finite dimensional complex projective space. Since the Euler characteristic of $\cx \P^D$ is $D + 1 > 0$, we get that $\Upsilon_t$ has to vanish somewhere. Let us choose $\Phi_t$ so that $\Upsilon_t \( \cx \Phi_t \) = 0$. Thus, for $t$ small enough, we have that
	\begin{equation}
		\cF_{\tau_0 + \epsilon \( \Phi, t \)} \( t \( \Phi_t + \Psi \( \Phi_t , t, \epsilon \( \Phi_t , t \) \) \) \) = 0,
	\end{equation}
	or, in other words, the pair
	\begin{equation}
		\( \nabla^0 + \cA \( t \Phi_t + \Psi \( \Phi_t, t, \epsilon \( \Phi_t , t \) \) \), t \Phi_t + \Psi \( \Phi_t, t, \epsilon \( \Phi_t , t \) \) \) \in \cC_\cL,
	\end{equation}
	is a (weak) solution to Ginzburg--Landau \cref{eq:gl1,eq:gl2} with $\tau_t \eqdef \tau_0 + \epsilon_t$. Defining
	\begin{align}
		\epsilon_t	&\eqdef t^{- 2} \epsilon \( \Phi_t , t \), \\
		\Psi_t		&\eqdef t^{- 2} \Psi \( \Phi_t, t, t^2 \epsilon_t \), \\
		\cA_t		&\eqdef t^{- 2} \cA \( t \Phi_t + t^3 \Psi_t \),
	\end{align}
	gives us the fields in \cref{eq:branch}, and the claims about the regularity of the branch of triples $\( \cA_t, \Psi_t, \epsilon_t \)$ have already been proven. This completes the proof.
\end{proof}

\smallskip

\begin{remark}
	More detailed analysis of $\cA_t, \Psi_t, \epsilon_t$ shows that if we set
	\begin{align}
		\cA_0		&\eqdef G \( j \( \nabla^0, \Phi_t \) \), \\
		\epsilon_0	&\eqdef \| \Phi_t \|_{L^4}^4 + \kappa^{- 2} \Re \( \left\langle \nabla^0 \Phi_t \middle| \cA_0 \Phi_t \right\rangle \) \\
		\Psi_0		&\eqdef - G_\lambda \( 2 \cA_0^* \( \nabla^0 \Phi_t \) + \kappa^2 \left| \Phi_t \right|^2 \Phi_t \),
	\end{align}
	then
	\begin{align}
		\cA_t		&= \cA_0 + O \( t \), \\
		\epsilon_t	&= \epsilon_0 + O \( t \), \\
		\Psi_t		&= \Psi_0 + O \( t \).
	\end{align}
\end{remark}

\smallskip

\begin{remark}
	When $X$ has nontrivial first de Rham cohomology, then one runs into the following problem: When eliminating the $1$-form one cannot use \cite{W04}*{Theorem~5.2} to get \cref{ineq:A_bound}. Instead, if $\Pi_{H_\dR^1}$ is the $L^2$-orthogonal projection from $\Omega_{\rd^*}^1$ onto $H_\dR^1 \( X, g \)$, then the harmonic part of $A \eqdef \nabla - \nabla^0$, which we denote by $A_H$, needs to satisfy the following equation in the small $t$ limit:
	\begin{equation}
		\Pi_{H_\dR^1} \( |\Phi|^2 A_H \) = \Pi_{H_\dR^1} \( i \Im \( h \( \nabla^0 \Phi, \Phi \) \) \) + O \( t^2 \).
	\end{equation}
	It is not obvious if $A_H$ can be chosen to be small.

	In \Cref{ss:proof_of_branch_2} we study a case where this issue can be circumvented, albeit at the cost of only having bifurcation at the bottom of the spectrum of the Laplacian and only getting a bifurcating (countable) sequence, as opposed to a (continuum) branch.
\end{remark}

\medskip

\subsection{Bifurcation of absolute minimizers on K\"ahler manifolds}
\label{ss:proof_of_branch_2}

In this section, let us assume that $\kappa^2 \geqslant \tfrac{1}{2}$, $\( X, g, \omega \)$ is a K\"ahler manifold of real dimension $N$, $\nabla^0$ is a Hermitian Yang--Mills connection on $\cL$, that is
\begin{equation}
	i \Lambda F_{\nabla^0} = f_0, \quad \& \quad F_{\nabla^0}^{0, 2} = 0. \label{eq:HYM}
\end{equation}
Assume furthermore that $\cL$ carries nontrivial holomorphic sections with respect to holomorphic structure induced by $\nabla^0$. As before, let $\Delta_0 = \( \nabla^0 \)^* \nabla^0$. For the rest of the paper, let $\cH^0 (X; \cL)$ be the space of holomorphic sections with respect to $\nabla^0$.

First, for sake of completeness, we prove a well-known fact about Hermitian Yang--Mills connections on line bundles over K\"ahler manifolds.

\begin{theorem}
	\label{theorem:HYMs}
	Let $f_0$ be as above. Then
	\begin{equation}
		i \Lambda F_{\nabla^0} = f_0 = \frac{2 \pi}{\Vol \( X, g \)} \( c_1 (\cL) \cup [\omega]^{n - 1} \) [X]. \label{eq:f_0_def}
	\end{equation}
	The space of Hermitian Yang--Mills connection on $\cL$, in Coulomb gauge with respect to $\nabla^0$, is an affine space over $H_{\dR}^1 \( X, g \)$ and thus compact modulo gauge (in fact, the moduli space is diffeomorphic to a $b_1 \( X \)$-dimensional torus).

	Let $\lambda$ be the smallest eigenvalue of $\Delta_0$. When $\cL$ carries nontrivial holomorphic sections with respect to holomorphic structure induced by $\nabla^0$, then
	\begin{equation}
		\lambda = f_0 \geqslant 0, \label{eq:def_lambda}
	\end{equation}
	and the corresponding eigenvectors are the holomorphic sections on $\cL$.

	In particular, when $\lambda = 0$, then the lowest eigenspace is one dimensional (over $\cx$) and is spanned by covariantly constant sections.
\end{theorem}

\begin{proof}[Proof sketch:]
	Since $F_\nabla$ is harmonic and, by Chern--Weil theory, $\left[ i F_\nabla \right] = \tfrac{2 \pi}{\Vol \( X, g \)} c_1 (\cL)$, we have
	\begin{align}
		f_0 &= \frac{1}{\Vol \( X, g \)} i \Lambda F_{\nabla^0} \Vol \( X, g \) \\
			&= \frac{1}{\Vol \( X, g \)} \int\limits_X i \Lambda F_\nabla \omega^n \\
			&= \frac{2 \pi}{\Vol \( X, g \)} \frac{1}{2 \pi} \int\limits_X i F_\nabla \wedge \omega^{n - 1} \\
			&= \frac{2 \pi}{\Vol \( X, g \)} \( c_1 (\cL) \cup [\omega]^{n - 1} \) [X],
	\end{align}
	which proves \cref{eq:HYM}.

	Using the K\"ahler--Weitzenb\"ock identity
	\begin{equation}
		\Delta_0 = 2 \delbar_{\nabla^0}^* \delbar_{\nabla^0} + i \Lambda F_{\nabla^0} = 2 \delbar_{\nabla^0}^* \delbar_{\nabla^0} + f_0,
	\end{equation}
	which shows that $\Spec \( \Delta_0 \) \subset \left[ f_0, \infty \right)$, and $\Delta_0 \uppsi = f_0 \uppsi$ exactly when $\delbar_{\nabla^0} \uppsi = 0$, that is, when $\uppsi$ is $\delbar_{\nabla^0}$-holomorphic.
\end{proof}

\smallskip

Next, we introduce a few important analytic tools and results. Following \cite{PPS21}*{Section~5.1}, let $p > N$ and let now $\cX$ denote the completion of smooth sections of $\cL$ via the norm
\begin{equation}
	\| \upphi \|_\cX \eqdef \| \nabla \upphi \|_{L^2} + \| \upphi \|_{L^p},
\end{equation}
which is slightly stronger than the one in \cref{eq:X_norm_def}. For each $\( \nabla, \upphi \) \in i \Omega_{\rd^*}^1 \times \cX$, let the \emph{modified Ginzburg--Landau energy} be
\begin{equation}
	\widetilde{\cE}_{\tau, \kappa} \( \nabla, \upphi \) \eqdef \tfrac{1}{2} \int\limits_X \( \left| F_\nabla \right|^2 + \left| \nabla \upphi \right|^2 + \tfrac{\kappa^2}{2} W_\tau \( \left| \upphi \right| \) \) \ \vol_g, \label{eq:mod_gle}
\end{equation}
where
\begin{equation}
	W_\tau (x) \eqdef \left\{ \begin{array}{ll} \( x^2 - \tau \)^2, & x^2 \leqslant \tau, \\ \( x^2 - \tau \)^p, & x^2 > \tau. \end{array} \right.
\end{equation}
Then the techniques of \cite{PS20}*{Section~7} can be used (almost verbatim, with trivial modifications) to prove the following Palais--Smale type property (see also \cite{PPS21}*{Section~5.1}):

\begin{theorem}
	\label{theorem:gPS}
	The function $\widetilde{\cE}_{\tau, \kappa} : \Conn_\cL \times \cX \rightarrow \rl_+$ is $C^1$.

	If $\( \nabla_n, \upphi_n, \tau_n, \kappa_n \)_{n \in \N}$ is a sequence in $i \Omega_{\rd^*}^1 \times \cX \times \rl_+^2$ such that
	\begin{enumerate}

		\item $\( \tau_n, \kappa_n \) \rightarrow \( \tau, \kappa \) \in \rl_+^2$, as $n \rightarrow \infty$.

		\item $\sup \( \left\{ \ \widetilde{\cE}_{\tau_n, \kappa_n} \( \nabla_n, \upphi_n \) \ \middle| \ n \in \N \ \right\} \) < \infty$.

		\item $\( D \widetilde{\cE}_{\tau_n, \kappa_n} \)_{\( \nabla_n, \upphi_n \)} \rightarrow 0 \in T^*\( \Conn_\cL \times \cX \)$, as $n \rightarrow \infty$.

	\end{enumerate}
	Then, there is $ \( \nabla, \upphi \) \in \Conn_\cL \times \cX$, such that (after picking a subsequence and applying an appropriate sequence of (smooth) gauge transformations) 
	\begin{align}
		\lim\limits_{n \rightarrow \infty} \( \nabla_n, \upphi_n \)								&= \( \nabla, \upphi \) , \\
		\( D \widetilde{\cE}_{\tau, \kappa} \)_{\( \nabla, \upphi \)}						&= 0 \in T_{\( \nabla, \upphi \)}^*, \\
		\lim\limits_{n \rightarrow \infty} \widetilde{\cE}_{\tau, \kappa} \( \nabla_n, \upphi_n \)	&= \widetilde{\cE}_{\tau, \kappa} \( \nabla, \upphi \).
	\end{align}
	Furthermore, the critical points of $\cE_{\tau, \kappa}$ and $\widetilde{\cE}_{\tau, \kappa}$ are the same.
\end{theorem}

\begin{remark}
	In \cite{PS20}*{Section~7} the authors deal with the gauge ambiguity coming from harmonic gauge transformations by introducing a gauge invariant Finsler structure on $\cX$ and then get the classical Palais--Smale property of $\widetilde{\cE}_{\tau, \kappa}$. Alternatively, here we use a gauged version of Palais--Smale property.
\end{remark}

\smallskip

We are ready to prove our last result.

\begin{theorem}
	\label{theorem:branch_2}
	Let $\lambda$ be given by \cref{eq:def_lambda}, $\kappa^2 \geqslant \tfrac{1}{2}$, and $\tau_0 \eqdef \tfrac{\lambda}{\kappa^2}$. Then irreducible solutions to the Ginzburg--Landau \cref{eq:gl1,eq:gl2} exist exactly when $\tau > \tau_0$.

	Furthermore, there exists a Hermitian Yang--Mills connection, $\nabla^0$, on $\cL$, such that $\cH^0 (X; \cL)$ is nontrivial, and a sequence of pairs
	\begin{equation}
		\( \( \Phi_n, t_n \) \in \cH^0 (X; \cL) \times \rl_+ \)_{n \in \N},
	\end{equation}
	such that for all $n \in \N$, $\| \Phi_n \|_{L^2} = 1$ and $\lim_{n \rightarrow \infty} t_n = 0$, and there is sequence of triples
	\begin{equation}
		\( \( A_n, \upphi_n, \tau_n \) \in \Omega_{\rd^*}^1 \times \cX \times \rl_+ \)_{n \in \N},
	\end{equation}
	of the form
	\begin{equation}
		\( A_n, \upphi_n, \tau_n \) = \( t_n^2 \cA_n, t_n \Phi_n + t_n^3 \Psi_n, \tfrac{\lambda}{\kappa^2} + t_n^2 \epsilon_n \),
	\end{equation}
	such that the family
	\begin{equation}
		\left\{ \( \cA_n, \Psi_n, \epsilon_n \) \right\}_{n \in \N},
	\end{equation}
	is bounded in $L_1^2 \times \cX \times \rl_+$, and for each $n \in \N$ the pair $\( \nabla^0 + A_n, \upphi_n \)$ is an irreducible solution to the Ginzburg--Landau \cref{eq:gl1,eq:gl2} with $\tau = \tau_n > \tau_0$.
\end{theorem}

\begin{proof}
	First we show that if $\tau \leqslant \tau_0$, then all critical points of $\cE_{\tau, \kappa}$ are normal phase solutions. We use proof by contradiction: let $\tau \leqslant \tau_0$, $\( \nabla, \upphi \)$ be a critical point of $\cE_{\tau, \kappa}$ such that $\upphi \neq 0$, and
	\begin{align}
		w	&\eqdef \tfrac{1}{2} \( \tau - |\upphi|^2 \), \\
		f	&\eqdef i \Lambda F_\nabla.
	\end{align}
	Using \cref{eq:gl2}, we get
	\begin{equation}
		\( \Delta + 2 \kappa^2 |\upphi|^2 \) w = - \tfrac{1}{2} \Delta |\upphi|^2 + 2 \kappa^2 |\upphi|^2 w = - \Re \( h \( \upphi, \nabla^* \nabla \upphi \) \) + |\nabla \upphi|^2 + 2 \kappa^2 |\upphi|^2 w = |\nabla \upphi|^2. \label{eq:delta_w}
	\end{equation}
	Maximum principle then yields $w > 0$, or, equivalently $|\upphi|^2 < \tau$ everywhere on $X$. Using \cref{eq:gl1}, the Bianchi identity, $\rd F_\nabla = 0$, and the K\"ahler identities, we get
	\begin{align}
		\Delta f	&= \rd^* \rd i \Lambda F_\nabla \\
					&= \rd^* [\rd, i \Lambda] F_\nabla \\
					&= \( \delbar^* + \del^* \) \( \delbar^* - \del^* \) F_\nabla \\
					&= \( \del^* - \delbar^* \) \( \delbar^* + \del^* \) F_\nabla \\
					&= \( \del^* - \delbar^* \) \( i \: \Im \( h \( \upphi, \nabla \upphi \) \) \) \\
					&= 2 \Re \( \del^* \( i \: \Im \( h \( \nabla \upphi, \upphi \) \) \)^{1, 0} \) \\
					&= \Re \( \del^* h \( \nabla^{0, 1} \upphi, \upphi \) - \del^* h \( \upphi, \nabla^{1, 0} \upphi \) \) \\
					&= \Re \( - h \( i \Lambda \( \nabla^{1, 0} \nabla^{0, 1} \upphi \), \upphi \) - |\nabla^{0, 1} \upphi|^2 - h \( \upphi, i \Lambda \( \nabla^{0, 1} \nabla^{1, 0} \upphi \) \upphi \) + |\nabla^{1, 0} \upphi|^2 \) \\
					&= - |\upphi|^2 f + |\nabla^{1, 0} \upphi|^2 - |\nabla^{0, 1} \upphi|^2.
	\end{align}
	Thus we proved the equation
	\begin{equation}
		\( \Delta + |\upphi|^2 \) f = |\nabla^{1, 0} \upphi|^2 - |\nabla^{0, 1} \upphi|^2. \label{eq:delta_f}
	\end{equation}
	Now, using \cref{eq:delta_w,eq:delta_f}, we get that
	\begin{equation}
		\( \Delta + |\upphi|^2 \) \( \kappa^2 \tau - |\upphi|^2 \pm f \) = \( \kappa^2 - \tfrac{1}{2} \) |\upphi|^4 + (1 \pm 1) |\nabla^{1, 0} \upphi|^2 + (1 \mp 1) |\nabla^{0, 1} \upphi|^2 \geqslant 0,
	\end{equation}
	thus, using the maximum principle again, we get that
	\begin{equation}
		|i \Lambda F_\nabla| = |f| < \kappa^2 \tau - |\upphi|^2 < \kappa^2 \tau.
	\end{equation}
	Using the homological invariance of the degree, we get that if there is an irreducible solution to the Ginzburg--Landau \cref{eq:gl1,eq:gl2} on $\( \cL, h \)$, then
	\begin{equation}
		\lambda	= \frac{1}{\Vol \( X, g \)} \int\limits_X \( i \Lambda F_\nabla \) \vol_g < \frac{1}{\Vol \( X, g \)} \kappa^2 \tau \Vol \( X, g \) = \kappa^2 \tau \leqslant \kappa^2 \tau_0 = \lambda,
	\end{equation}
	which is a contradiction. Thus we proved that if $\tau \leqslant \tau_0$, then all critical points of $\cE_{\tau, \kappa}$ are normal phase solutions.

	Now, on the one hand, using \cref{theorem:gPS}, we get that for all $\tau > 0$, there are absolute minimizer for $\cE_{\kappa, \tau}$. Let $\epsilon > 0$ and $\tau \eqdef \tau_0 + \epsilon$. On the other hand, consider $\Phi \in \ker \( \Delta_0 - \lambda \id \)$ with unit $L^2$-norm. There exists $C > 0$, independent of the choice of $\Phi$, such that $\| \Phi \|_{L^4} \leqslant C$. Thus if $s \in \( 0, \tfrac{\sqrt{2 \epsilon}}{C^2} \)$, then
	\begin{align}
		\cE_{\kappa, \tau} \( \nabla^0, s \: \Phi \) - \cE_{\kappa, \tau} \( \nabla^0, 0 \)	&= s^2 \| \nabla^0 \Phi \|_{L^2}^2 - \( \kappa^2 \tau_0 + \kappa^2 \epsilon \) \| \Phi \|_{L^2}^2 + s^4 \tfrac{\kappa^2}{2} \| \Phi \|_{L^4}^4 \\
			&\leqslant s^2 \( \lambda - \kappa^2 \tau_0 - \kappa^2 \epsilon \) + \tfrac{C^4 \kappa^2}{2} s^4 \\
			&\leqslant s^2 \kappa^2 \( \tfrac{C^4}{2} s^2 - \epsilon \) \\
			&< 0.
	\end{align}
	Since the energy of all normal phase solutions are the same, the absolute minimum is not achieved at a normal phase solution, and hence $\upphi_\epsilon \neq 0$. Let $\( \epsilon_n \)_{n \in \N}$ be a positive sequence that converges to zero, $\tau_n \eqdef \tau_0 + \epsilon_n$, and let $\( \nabla_n, \upphi_n \)$ be the corresponding minimizer. Using again \Cref{theorem:gPS}, we get that (after picking a subsequence and changing gauge) $\( \nabla_n, \upphi_n \)$ converges to a critical point of $\cE_{\kappa, \tau_0}$ and since $\tau_n \rightarrow \tau_0$, that this critical point has the form $\( \nabla^0, 0 \)$. Let us write $\nabla_n \eqdef \nabla^0 + t_n^2 \cA_n$ and $\upphi_n \eqdef t_n \Phi_n + \Psi_n$, where $\Phi_n \in \ker \( \( \nabla^0 \)^* \nabla^0 - \lambda \id \)$ has unit $L^2$-norm and $\Psi_n \perp_{L^2} \ker \( \( \nabla^0 \)^* \nabla^0 - \lambda \id \)$. It is straightforward to prove that the sequence of $5$-tuples, $\left\{ \( \cA_n, \Phi_n, \Psi_n, \epsilon_n, t_n \) \right\}_{n \in \N}$, satisfies the claims of the theorem. This completes the proof.
\end{proof}

	\bibliography{references}

\end{document}